\documentclass[10pt]{amsart}

\usepackage{latexsym}
\usepackage{amssymb}
\usepackage{amsrefs}
\usepackage{mathrsfs}
\usepackage{amsmath}
\usepackage{fancybox,color}
\usepackage{enumerate}
\usepackage[active]{srcltx}
\usepackage[colorlinks, linkcolor=black, citecolor=black]{hyperref}
\usepackage{graphicx}

\usepackage{array}
\newcolumntype{P}[1]{>{\centering\arraybackslash}p{#1}}
\newcolumntype{M}[1]{>{\centering\arraybackslash}m{#1}}

\numberwithin{equation}{section}

\newcommand{\E}{{\mathbb E}}

\def\1{\raisebox{2pt}{\rm{$\chi$}}}

\newtheorem{theorem}{Theorem}[section]
\newtheorem{corollary}[theorem]{Corollary}
\newtheorem{lemma}[theorem]{Lemma}
\newtheorem{proposition}[theorem]{Proposition}

\theoremstyle{definition}
\newtheorem{remark}[theorem]{Remark}
\newtheorem{example}[theorem]{Example}

 \def\1{\raisebox{2pt}{\rm{$\chi$}}}
 
\newcommand{\pare}[1]{\left( #1 \right)} 

\newcommand{\fd}[1]{\Delta #1 } \newcommand{\sfd}[1]{\Delta^2 #1 } \newcommand{\kfd}[2]{\Delta^{#1} #2}

\def\vint_#1{\mathchoice%
          {\mathop{\kern 0.2em\vrule width 0.6em height 0.69678ex depth -0.58065ex
                  \kern -0.8em \intop}\nolimits_{\kern -0.4em#1}}%
          {\mathop{\kern 0.1em\vrule width 0.5em height 0.69678ex depth -0.60387ex
                  \kern -0.6em \intop}\nolimits_{#1}}%
          {\mathop{\kern 0.1em\vrule width 0.5em height 0.69678ex
              depth -0.60387ex
                  \kern -0.6em \intop}\nolimits_{#1}}%
          {\mathop{\kern 0.1em\vrule width 0.5em height 0.69678ex depth -0.60387ex
                  \kern -0.6em \intop}\nolimits_{#1}}}
\def\vintslides_#1{\mathchoice%
          {\mathop{\kern 0.1em\vrule width 0.5em height 0.697ex depth -0.581ex
                  \kern -0.6em \intop}\nolimits_{\kern -0.4em#1}}%
          {\mathop{\kern 0.1em\vrule width 0.3em height 0.697ex depth -0.604ex
                  \kern -0.4em \intop}\nolimits_{#1}}%
          {\mathop{\kern 0.1em\vrule width 0.3em height 0.697ex depth -0.604ex
                  \kern -0.4em \intop}\nolimits_{#1}}%
          {\mathop{\kern 0.1em\vrule width 0.3em height 0.697ex depth -0.604ex
                  \kern -0.4em \intop}\nolimits_{#1}}}

\newcommand{\aveint}[2]{\mathchoice%
          {\mathop{\kern 0.2em\vrule width 0.6em height 0.69678ex depth -0.58065ex
                  \kern -0.8em \intop}\nolimits_{\kern -0.45em#1}^{#2}}%
          {\mathop{\kern 0.1em\vrule width 0.5em height 0.69678ex depth -0.60387ex
                  \kern -0.6em \intop}\nolimits_{#1}^{#2}}%
          {\mathop{\kern 0.1em\vrule width 0.5em height 0.69678ex depth -0.60387ex
                  \kern -0.6em \intop}\nolimits_{#1}^{#2}}%
          {\mathop{\kern 0.1em\vrule width 0.5em height 0.69678ex depth -0.60387ex
                  \kern -0.6em \intop}\nolimits_{#1}^{#2}}}

\newcommand\JPB{\textcolor{black}}

\def\ln{ \log }

\begin{document}
\title[Secretary Problem with quality-based payoff]{Secretary Problem with quality-based payoff}
\author[Blanc]{Pablo Blanc}
\address{Departamento de Matem\'atica, Facultad de Ciencias Exactas y Naturales, Universidad de Buenos Aires and IMAS, CONICET, Argentina}
\email[P. Blanc]{pblanc@dm.uba.ar}
\thanks{PB, JPB and DK were partially supported by a CONICET doctoral fellowship}

\author[Borthagaray]{Juan Pablo Borthagaray}
\email[J. P. Borthagaray]{jpbortha@dm.uba.ar}

\author[Kohen]{Daniel Kohen}
\email[D. Kohen]{dkohen@dm.uba.ar}

\author[Mereb]{Mart\'in Mereb}
\email[M. Mereb]{mmereb@gmail.com}

\keywords{Secretary problem, probability, order statistics}
\subjclass[2010]{60G40, 62L15}

\begin{abstract} 
We consider a variant of the classical Secretary Problem. 
In this setting, the candidates are ranked according to some exchangeable random variable and
the quest is to maximize the 
expected quality
of the chosen aspirant.
We find an upper bound for the optimal hiring rule, 
present examples showing it is sharp, and
recover the classical case, among other results.
\end{abstract}

\maketitle
\section*{Introduction}
A recruiter is faced with the task of selecting the best assistant 
among a stream of $ n $ applicants, on a reject-or-hire basis. Namely, the decision is made right after
the interview, there is no coming back once a candidate is rejected, and 
all the information gathered during the interview is whether the current postulant
is or not better than all its precursors.
The best strategy for the interviewer implies establishing 
a \emph{threshold index}  and selecting the first candidate that arrives after such a point and is the best the recruiter has interviewed so far.
In the classical Secretary Problem the question is to find the optimal threshold index,  provided that we want to
maximize the probability of hiring the best aspirant.

In this nice introductory example of a statistical decision making problem one learns that the best strategy
 is to blindly reject the first $ \approx n/e $ candidates and from that point on, 
to select the first postulant that is superior than all of the previous ones. 
If none is chosen with this plan, one just hires the last applicant. 
We consider $ n $, the total number of candidates, as a known quantity, that all of them are totally ranked 
with no ties, the recruiter is only allowed to determine 
if the current aspirant is the best that has arrived so far,
and that the order in which applicants arrive is random.

For a historical overview of this problem, some generalizations, and a conjecture about Kepler's choice of his second wife, see the interesting article by Ferguson ~\cite{Fer}. \JPB{To the best of the authors' knowledge, the first published solution  of the classical Secretary problem is due to Lindley \cite{Lin}; the problem of minimizing the expected rank of the candidate selected was studied by Chow et al. \cite{Chow}.}

\JPB{One of the first variants analyzed in the literature is the full-information problem, i.e., when the recruiter is able to gather information about the distribution and value of the candidates. This problem was solved by Gilbert and Mosteller \cite{GilMos}. Interestingly, the problem of minimizing the expected rank of the selected candidate with full information --known as the Robbins' problem-- remains open. Nonetheless, some bounds for this fully history-dependent problem are known; see \cite{Bruss}.}

Among other well-known variants of the classical Secretary Problem there is the Post-doc Problem ~\cite{vanderbei1995postdoc}, under the assumption that success 
is achieved when the selected applicant is the second best, 
(considering that the best one will go to Harvard anyway); the
Problem of admitting a Class of Students, instead of only one, from an aspirant pool,
in which the task is to find a subset of candidates all of them better ranked than all the rejected ones in an
on-line algorithm ~\cite{Van}; 
and the Problem of selecting the best $ k $ secretaries out of $ n $ with a similar method ~\cite{Gird}.
The Odds-theorem \cite{Bru} provides another framework in which the classical Secretary Problem may be solved and also allows to handle with group interviews \cite{matsui2016lower}, \cite{tamaki2010sum}.

In the present article we tackle a variant of the Secretary Problem in which the goal is to maximize the 
expected value of the 
quality
of the selected applicant. Bearden~\cite{Bea} proves that the
optimal
threshold index for independent and identically distributed (i.i.d.) \emph{uniform} 
random variables
is close to $ \sqrt{n} .$

As a by-product of a fruitful discussion after a talk about the work
of Bearden, we consider different distributions and prove that, 
in general, the optimal threshold index $ c_{*}(n) $ is essentially bounded from above by $ n/e $ 
(see Theorem~\ref{teo:bound} below for details). 
We provide several examples for both continuous and discrete distributions, 
and study the behavior of the optimal threshold index in each situation. 
Among these examples, we recover the classical Secretary Problem (Example \ref{ex:classic}) and show that the bound from Theorem ~\ref{teo:bound} is sharp (see examples \ref{ex:pareto}, \ref{ex:classic} and \ref{ex:bernoulli}). 
We also prove that the order statistics form a complete monotonic sequence.

The paper is organized as follows. In Section~\ref{s:descprob} we give the general setting of the problem, and define the optimal threshold index as well, in Section~\ref{s:mainth}
we state and prove the main result, namely the upper bound of the optimal stopping rule regardless of the distribution chosen (Theorem~\ref{teo:bound}).
Afterwards, in Section~\ref{s:othres}, we 
study the behavior of
$ c_{*}(n) $ as $ n\to \infty ,$ and finally in Section~\ref{s:examp} we provide several examples as 
Exponential, Normal, Pareto distributions, together with permutations and Bernoulli variables, and 
explain how to recover the classical problem from our framework.

This work was originated in the inconspicuous  ``$2038-$seminar'' which provided an excellent work atmosphere.
The authors want to thank IMAS-CONICET and DM-FCEyN-UBA for their support.
Special gratitude is due to their secretaries and the people who hired them as well, since they seemed 
to have been aware of these results beforehand.

\section{Description of the problem}\label{s:descprob} 
In this paper, we consider the following variant of the secretary problem. A recruiter wants to hire an assistant among $n$ candidates, under these conditions:
\begin{itemize}
	\item The qualifications of the applicants are given by exchangeable random variables $(X_{k})_{k=1,\dots,n}$ with finite expected value.
	\item The recruiter is \emph{a priori} aware of their joint distribution.
	\item The candidates start arriving to the interview, one by one, and the only information the recruiter is able to gather is whether the current applicant is the best one evaluated so far. In particular, the interviewer is not aware of the candidate's actual quality.
	\item Once the interview is finished, the recruiter has to choose whether rejecting or hiring the applicant. If an applicant is rejected, it is not possible to recall them later.
	\item The goal of the recruiter is to maximize the expected quality of the chosen candidate.
\end{itemize} 
\begin{remark}
It is necessary to define properly what happens if the current applicant is as good as 
the most qualified interviewed by then.
To this end, we rank the candidates by quality, break ties at random 
and regard a candidate as unsurpassed thus far if he or she is the best one according to this rank.
\end{remark}

Our first goal is to characterize the possible optimal strategies for the recruiter.
First, note that if a candidate is not the best qualified on arrival, their expected value given that information is lower than the expected value of a randomly chosen applicant.
Hence, any optimal strategy should only consider selecting a candidate if it is the best so far (except for the last one which we are forced to hire).

Let us define
\[
\mu_k:=\E[X_k|X_k=\max\{X_1,\dots,X_k\}]
\]
and
\begin{align*}
\mu:=\E[X_k].
\end{align*}
These values are well defined and are finite since the random variables are assumed to be identically distributed 
and to have finite expected value.
Moreover, since the variables are exchangeable, $\mu_k=\E[\max\{X_1,\dots,X_k\}]$ and thus $\mu_{k+1}\geq \mu_k$.

The following result shows that the possible optimal strategies for the recruiter follow the same pattern as in the classical Secretary Problem \cite[p.48]{Lin} (and most of its variants).

\begin{proposition}
The optimal strategy for the recruiter consists of choosing an adequate  $c \in \{1,\ldots, n\}$, to blindly reject the first $c-1$ candidates
and from that point on to hire a candidate if, when it comes to the interview, 
it is the best so far. 
\end{proposition}

\begin{proof}
Let $\nu_k$ be the expected value of the candidate optimally chosen given that we rejected the first $k$ candidates. 
We have that $\nu_{k+1}\leq \nu_k$ .
Recall that a candidate should only be considered if it is the best so far at arrival.
Given that the $k$-th candidate is the best so far, the recruiter 
should hire him only if $\mu_k\geq \nu_k$.
Observe that $\mu_{n-1}\geq \mu_1 = \nu_{n-1}$ and hence the set of values such that $\mu_k\geq \nu_k$ is not empty.

Since $\mu_k$ is increasing in $k$ and $\nu_k$ is decreasing, $\mu_k\geq \nu_k$ implies $\mu_{k+1}\geq \nu_{k+1}$.
Thus the recruiter should follow the strategy described in the statement for $c=min\{k:\mu_k\geq \nu_k\}$.
\end{proof}

Let $V_n(c)$ be the expected value of the hired candidate when following the strategy described in the proposition.
Our goal is to maximize this quantity among all possible values of $c$.
We define the \emph{optimal threshold index} $c_*(n)$, as the smallest value of $c$ that maximizes $V_n(c)$.
Hence, the optimal strategy for the recruiter consists of blindly rejecting the first $c_*(n)-1$ candidates
and from that point on to hire a candidate if, when it comes to the interview, it is the best so far. 

In order to find the optimal threshold index, we compute the expected value $V_n(c)$ of the hired candidate for each value of $c$ as follows. We set  $V_n(1):=\E[X_1]$, and given $c\geq 2$ we set
\begin{equation}
V_n(c)=\sum_{k=c}^{n-1}\mu_k\frac{c-1}{k-1}\frac{1}{k}+\frac{c-1}{n-1}\mu.
\label{eq:def_vn}
\end{equation}

Indeed, we are summing the expected value for the $k$-th applicant (provided it surpasses all the former ones) 
times the probability of being selected, and including at the end of this formula the case
in which the recruiter hires the last one.
Observe that the $k$-th candidate is selected if and only if the best one among the first $k-1$ is 
in the group of the first $c-1$ applicants and the $k$-th is superior to its precursors.



The discrete derivative of $V_n,$ defined by the forward difference, has the following expression
\begin{equation}\label{eq:derivada1}
\fd{V}_n(c) := V_n(c+1)-V_n(c)=\sum_{k=c+1}^{n-1}\frac{\mu_k}{k(k-1)}-\frac{\mu_c}{c}+\frac{\mu}{n-1}
\end{equation}
valid for $ c \geq 1. $
We aim to find an expression for $V_n(c+1)-V_n(c)$ suitable for algebraic manipulations.
In order to do so, let us recall the summation by parts formula
\[
\sum_{k=m}^n f_k(g_{k+1}-g_k) = \left[f_{n+1}g_{n+1} - f_m g_m\right] - \sum_{k=m}^n g_{k+1}(f_{k+1}- f_k).
\]
This enables us to rewrite \eqref{eq:derivada1} as
\begin{equation}\label{eq:derivada2}
\fd{V}_n(c)=\sum_{k=c}^{n-1}\frac{\mu_{k+1}-\mu_k}{k}-\frac{\mu_{n}-\mu_1}{n-1}.
\end{equation}
Therefore, the discrete second derivative of $V_n$ becomes
\begin{align*}
\sfd{V}_n(c) := V_n(c+2)-2V_n(c+1) +V_n(c) 
=-\frac{\mu_{c+1}-\mu_{c}}{c}
\le 0.
\end{align*}
The last inequality holds because $\mu_k \le \mu_{k+1}$ for every $k$.
This proves that $ V_n$ is concave in $ c $ and therefore, local maxima are automatically global maxima.
In conclusion, the desired $c$ is the first one that satisfies  $\fd{V}_n(c)\le 0$ and $\fd{V}_n(c-1) \ge 0$. 
We summarize the preceding discussion in the following proposition.

\begin{proposition}\label{d:optthre}
The \emph{optimal threshold index} for $n$ secretaries satisfies $$c_*(n) := \min \{ c \ge 1 \colon \fd{V}_n(c) \le 0 \}.$$
\end{proposition}

\begin{remark}
It could be the case that there exists more than one $c$ maximizing $V_n(c)$.
Consequently, whenever we set an upper bound for the threshold index we are stating that there exists a $c$ in the set of maximizers of $V_n$ that is less or equal than the bound, and
whenever we provide a lower bound, it means that every maximizer is greater or equal than the bound.

Setting $c_*$ as the lowest possible is natural when considering the problem from the point of view of the recruiter,
who is seeking to maximize the quality of the selected candidate and not to perform too many interviews.
\end{remark}


\section{Main theorem}\label{s:mainth}
After having defined the optimal threshold index, the reader may ask himself about the dependence of $c_*(n)$
on the distribution of the random variables $(X_{k})_{k=1,\dots,n}$. 
It is well known that for the classical Secretary Problem the optimal threshold index is given by $n/e$. 
We claim that this number is an upper bound for $c_*(n)$, regardless of the distributions 
$(X_{k})_{k=1,\dots,n}$ of the candidates. More precisely, in this Section we prove the following theorem.

\begin{theorem}
\label{teo:bound}
Given any set of exchangeable random variables $(X_{k})_{k=1,\dots,n}$, the optimal threshold index $c_*(n)$ satisfies 
\begin{equation} \label{eq:cota_c}
\sum_{k=c_*(n)-1}^{n-1}\frac{1}{k} > 1.  
\end{equation}
In particular, the bound $c_*(n) \lesssim n/e$ holds.
\end{theorem}
In order to prove the previous theorem, let us define 
\begin{align*}
\mu_{k:n}:=\E[X_k|X_1\leq\dots\leq X_n].
\end{align*}
Since the random variables under consideration are exchangeable, it is clear that $\mu_k=\mu_{k:k}.$
Given the sequence $(\mu_{k})_{k=1,\dots,n}$, consider the discrete derivatives 
\begin{align*}
\fd{\mu_k}=\kfd{1}{\mu_k}=\mu_{k+1}-\mu_k,
\end{align*}
and in general
\begin{align*}
\kfd{j}{\mu_k}=\kfd{j-1}{\mu_{k+1}}-\kfd{j-1}{\mu_k}.
\end{align*}
Straightforward computations lead to the identity
\begin{align} \label{eq:derivada_j}
\kfd{j}{\mu_k}=\sum_{i=0}^{j} \mu_{k+i} {j \choose {i}} (-1)^{j-i}=\sum_{i=k}^{k+j} \mu_{i} {j \choose {i-k}} (-1)^{j-(i-k)}, \text{ if } k + j \le n.
\end{align}

Now we prove the complete monotonicity of the order statistics.
\begin{proposition} \label{prop:derivadadiscreta}
Let $k,j \ge 0$ be such that $k+j\le n$, then the following formula holds:
\[\kfd{j}{\mu_k}={k+j \choose j}^{-1} (-1)^{j+1} \left( \mu_{k+1:k+j}-\mu_{k:k+j}\right) .\]
\end{proposition}

\begin{proof}
If we consider exchangeable random variables $(X_{l})_{l=1,\dots,k+l}$, then every relative order between them is equally likely.
By considering the possible ranks of the top-ranked variable among the first $i$ when taking into account the relative order of all the variables, we obtain
\begin{align*}
\mu_i={k+j \choose i}^{-1}\:\sum_{l=i}^{k+j}\mu_{l:k+j}{l-1 \choose i-1}.
\end{align*} 

Next, we replace this expression on the right hand side of 
\eqref{eq:derivada_j}.
After interchanging the order of summation and some algebraic manipulation, we obtain
\begin{align}\label{eq:derivada_jjj}
\kfd{j}{\mu_k}=
\frac{j!(-1)^j}{(k+j)!}\sum_{l=k}^{k+j}\mu_{l:k+j}(l-1)!\sum_{i=0}^{l-k}\frac{(k+i)(-1)^i}{(l-k-i)!i!} .
\end{align}
Let us recall that
\begin{align*}
\sum_{i=0}^{t}(-1)^i{t \choose i}=\delta_t \quad \text{and} \quad \sum_{i=0}^{t}i(-1)^i{t \choose i}= -\delta_{t-1}.
\end{align*}
Therefore, the last summation of~\eqref{eq:derivada_jjj} becomes
\begin{align*}
\begin{split}
\sum_{i=0}^{l-k}\frac{(k+i)(-1)^i}{(l-k-i)!i!}
&=\frac{1}{(l-k)!} \left( k\sum_{i=0}^{l-k}(-1)^i{l-k \choose i} +\sum_{i=0}^{l-k}i(-1)^i{l-k \choose i} \right)\\
&=\frac{1}{(l-k)!} \left( k\delta_{l-k} -\delta_{l-k-1} \right).
\end{split}
\end{align*}
Plugging this last expression into~\eqref{eq:derivada_jjj} we get
\begin{align*}
\begin{split}
\kfd{j}{\mu_k}
&=\frac{j!(-1)^j}{(k+j)!}\left(
\mu_{k:k+j}(k-1)!k-\mu_{k+1:k+j}k!
\right) ,
\end{split}
\end{align*}
from which the result follows.
\end{proof}

An immediate consequence of the previous proposition is the following.

\begin{corollary}\label{co:totomonot}
Let $j+k\le n$. Then, $\kfd{j}{\mu_k}\geq 0$ if $j$ is odd and $\kfd{j}{\mu_k}\leq 0$ if $j$ is even.
\end{corollary}

Considering $j=1$ in Corollary \ref{co:totomonot}, we obtain the elementary result that the sequence $( \mu_k )_{k=1 \ldots n}$ is increasing. Moreover, setting $j=2$ we conclude:
\begin{corollary} \label{decreasing}
 The sequence $( \mu_{k+1}-\mu_{k} )_{k=1 \ldots n-1}$ is decreasing. 
\end{corollary}

At this point we are ready to provide a proof of our main result.
\begin{proof} [Proof of Theorem \ref{teo:bound}]
Replacing $ \mu_{n}-\mu_1 $ by $$  \sum_{k=1}^{n-1}\mu_{k+1}-\mu_k  $$ in~\eqref{eq:derivada2} gives
\begin{align*}
\fd{V}_n(c) =\sum_{k=c}^{n-1}(\mu_{k+1}-\mu_k)\left(\frac{1}{k}-\frac{1}{n-1}\right)
+\sum_{k=1}^{c-1}(\mu_{k+1}-\mu_k)\left(\frac{-1}{n-1}\right),
\end{align*}
whose right hand side is not greater than 
\begin{align*}
\sum_{k=c}^{n-1}(\mu_{c+1}-\mu_c)\left(\frac{1}{k}-\frac{1}{n-1}\right)+\sum_{k=1}^{c-1}(\mu_{c+1}-\mu_c)\left(\frac{-1}{n-1}\right),
\end{align*}
from which we conclude
\begin{align*}
\fd{V}_n(c)\leq(\mu_{c+1}-\mu_c)\left(\sum_{k=c}^{n-1}\frac{1}{k} -1 \right).
\end{align*}

Then, $\fd{V}_n(c)\leq 0$ when
\begin{align*}\sum_{k=c}^{n-1}\frac{1}{k} \leq 1,
\end{align*}
which proves that $c_*(n)$ satisfies property \eqref{eq:cota_c}.
\end{proof}

\begin{remark} 
The upper bound given by \eqref{eq:cota_c} is sharp, see examples \ref{ex:pareto}, \ref{ex:classic} and \ref{ex:bernoulli} below. 
On the other hand, there are no non-trivial  lower bounds for the optimal threshold index valid in general for any i.i.d. random variables. 
The idea behind this fact is that if almost every candidate has maximal quality, then the recruiter has no need to wait. 
We refer to Example \ref{ex:bernoulli} and Remark \ref{r:extbeh} for a simple construction in which the optimal threshold index is $c_* = 2$ for any number of applicants. 
\end{remark}

\section{Asymptotic results for independent variables}\label{s:othres}
Throughout this section we make the further assumption that for each $n$ the random 
variables $(X_{k})_{1\leq k\leq n}$ 
are independent with a given distribution. 
We disregard the case of a Dirac delta distribution in which all the candidates are equally suitable, 
and thus every strategy furnishes the same result.

In this setting it is interesting to study the behavior of $c_*(n)$ as $n$ varies.
We prove that $c_*$ is a non-decreasing function of $n$ that diverges as $n$ goes to infinity.
This means that the recruiter has to wait longer as $n$ grows and that $c_*(n)$ becomes as 
large as wanted.
Namely, given any $m\geq 1,$ the interviewer would have to reject the first $m$ candidates if $n$ 
(the total number of applicants) is taken large enough. In the classical case, this is a straightforward
 consequence of the odds-theorem and the odds-algorithm from~\cite{Bru}. 
 Our approach is elementary in nature.

\begin{proposition} The optimal threshold index grows with the amount of candidates, that is,
 $c_*(n) \leq c_*(n+1)$.
\end{proposition}
\begin{proof}
Given $c < c_*(n)$, using \eqref{eq:derivada2} we obtain
$$
0 < \fd{V}_n(c) = \sum_{k=c}^{n-1}\frac{\mu_{k+1}-\mu_k}{k}-\frac{\mu_{n}-\mu_1}{n-1} .
$$
Employing again \eqref{eq:derivada2},   now for $c$ and $n+1$ we obtain
\begin{align*}
\fd{V}_{n+1}(c)& = \fd{V}_n(c) + \frac{\mu_{n+1}-\mu_n}{n}  + \frac{\mu_{n}-\mu_1}{n-1}  - \frac{\mu_{n+1}-\mu_1}{n}  \\
& = \fd{V}_n(c)  + \frac{\mu_{n}-\mu_1}{n(n-1)} > 0 .
\end{align*}
Then, $\fd{V}_{n+1}(c) > 0$, so $c < c_*(n+1)$, implying that $c_*(n) \leq c_*(n+1)$ as claimed.
\end{proof}

\begin{lemma} \label{lemma:strict}
The sequence $(\mu_k)_{k=1}^\infty$ is strictly increasing.
\end{lemma}
\begin{proof}
If the Cumulative Distribution Function (CDF) of the variables is $F(x)$, then  the CDF of the maximum of $k$  is $F^k(x)$. Therefore,
\[\mu_{k+1}-\mu_k= \int_{-\infty}^{\infty} x \, d(-F^{k}(x)(1-F(x))) .\]

Integrating by parts the last expression we obtain 
\[\mu_{k+1}-\mu_k= \int_{-\infty}^{\infty} F(x)^k (1-F(x))dx  \]
with boundary terms vanishing thanks to the Lemma on page $37$ of \cite{DN}.

Since the variables are not Deltas, the last integral is positive and the result follows.
\end{proof}

\begin{remark}
The point of interest in the previous lemma is that not only is the sequence $(\mu_k)_{k=1}^\infty$ monotone, 
but it is \emph{strictly} increasing. To obtain this result, we require the random variables to be independent. For example, if we consider a set of $n$ applicants where $j$ are valued $0$ and $n-j$ are valued $1$, then $\mu_k = 1$ for every $k \ge j+1$. 
\end{remark}

\begin{proposition}\label{prop:c_diverge}
The optimal threshold index diverges with the number of candidates, that is,
$c_*(n) \to\infty$ when $n\to\infty$.
\end{proposition}
\begin{proof}
Let $c>0$ be fixed. From \eqref{eq:derivada1}, we have 
\[ \begin{split}
\fd{V}_n(c) & \ge \frac{\mu_c}{c(c+1)} + \mu_{c+1} \sum_{k=c+2}^{n-1} \frac{1}{k(k-1)} - \frac{\mu_c}{c} + \frac{\mu_1}{n-1} \\
& = \frac{\mu_{c+1} - \mu_c }{c+1} - \frac{\mu_{c+1} - \mu_1 }{n-1}.
\end{split} \]
Lemma \ref{lemma:strict} ensures that the first term above is strictly positive, while the last one tends to $0$ as $n$ tends to infinity, and so the derivative at $c$ is strictly positive for $n$ large enough.
\end{proof}

\section{Examples}\label{s:examp}
In this section we work through different families of distributions.
The problem of finding the optimal threshold index is invariant under linear scalings of the applicants' qualities.
For this reason the mean and variance of the random variables under consideration 
play no role in the estimates we provide.

In \S\ref{ss:iidcont} we deal with some continuous distributions (Exponential, Normal and Pareto), 
and manage to prove that the upper bound from Theorem~\ref{teo:bound} is asymptotically optimal in a precise sense.
In \S\ref{ss:disc} we work out some discrete examples, recover the solution of the classical Secretary Problem and exhibit an example that shows that
there is no non-trivial lower estimate for $c_*(n)$.

Plots of several of the examples considered, both for continuous and discrete distributions, are displayed in Figure \ref{fig:graficos}. These illustrate the different behaviors $c_*(n)$ may exhibit.

\begin{figure} 
\centering
\begin{tabular}{M{5.8cm}M{5.8cm}}
 \includegraphics[width=55mm]{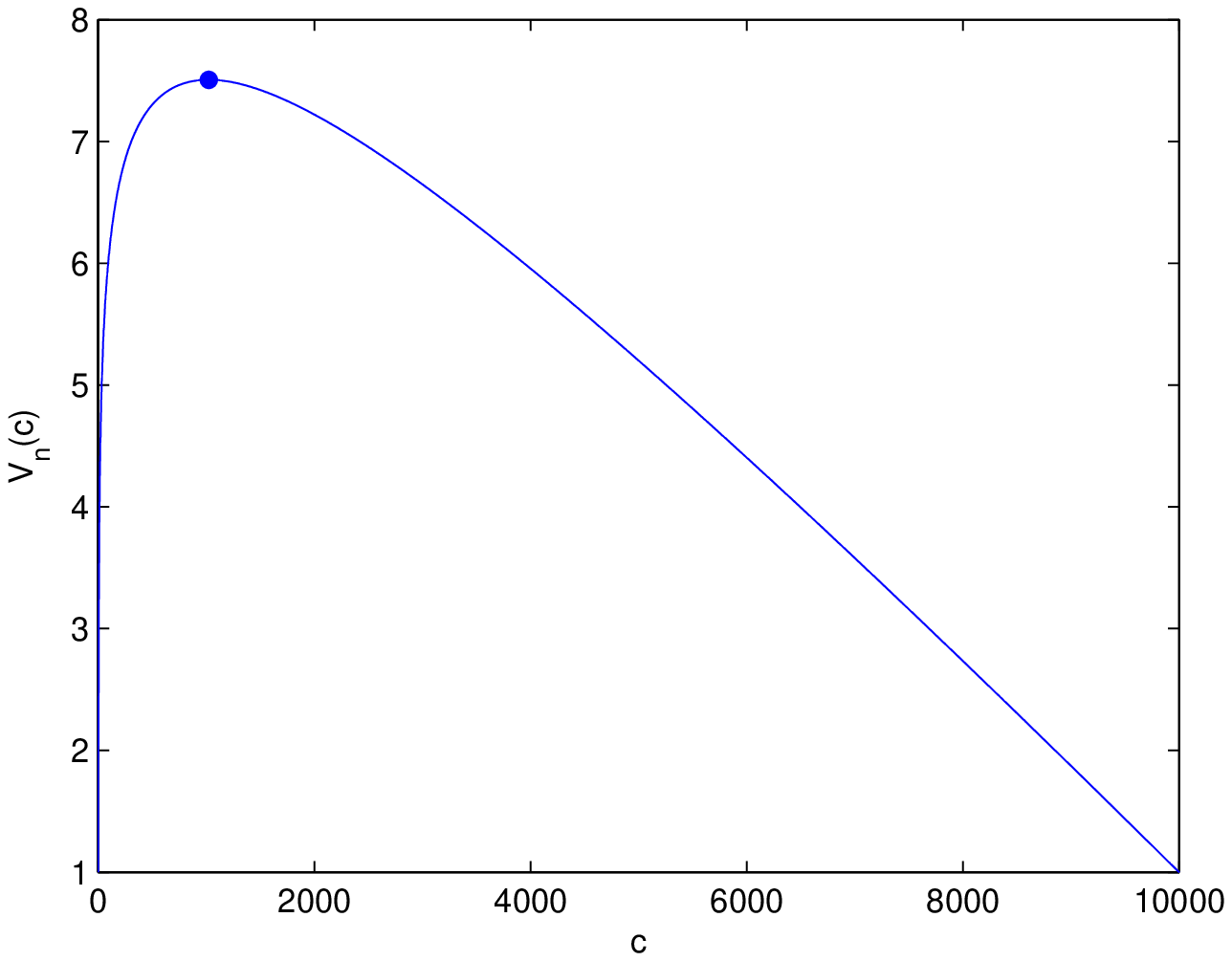} &   \includegraphics[width=55mm]{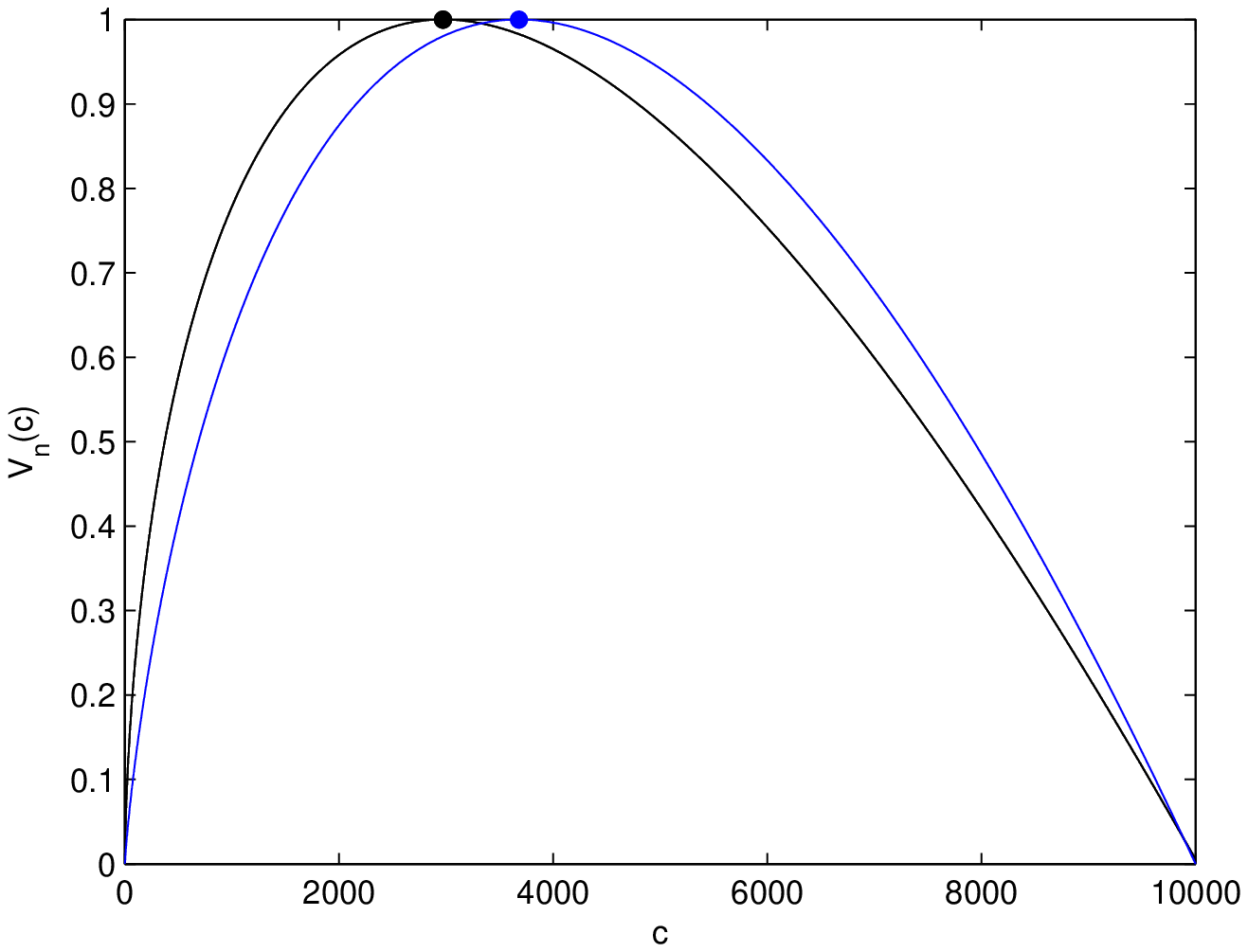} \\
{\footnotesize \centering (a) Exponential distribution.} & 
{\footnotesize \centering (b) Pareto distribution.}\\
\multicolumn{2}{c}{\includegraphics[width=55mm]{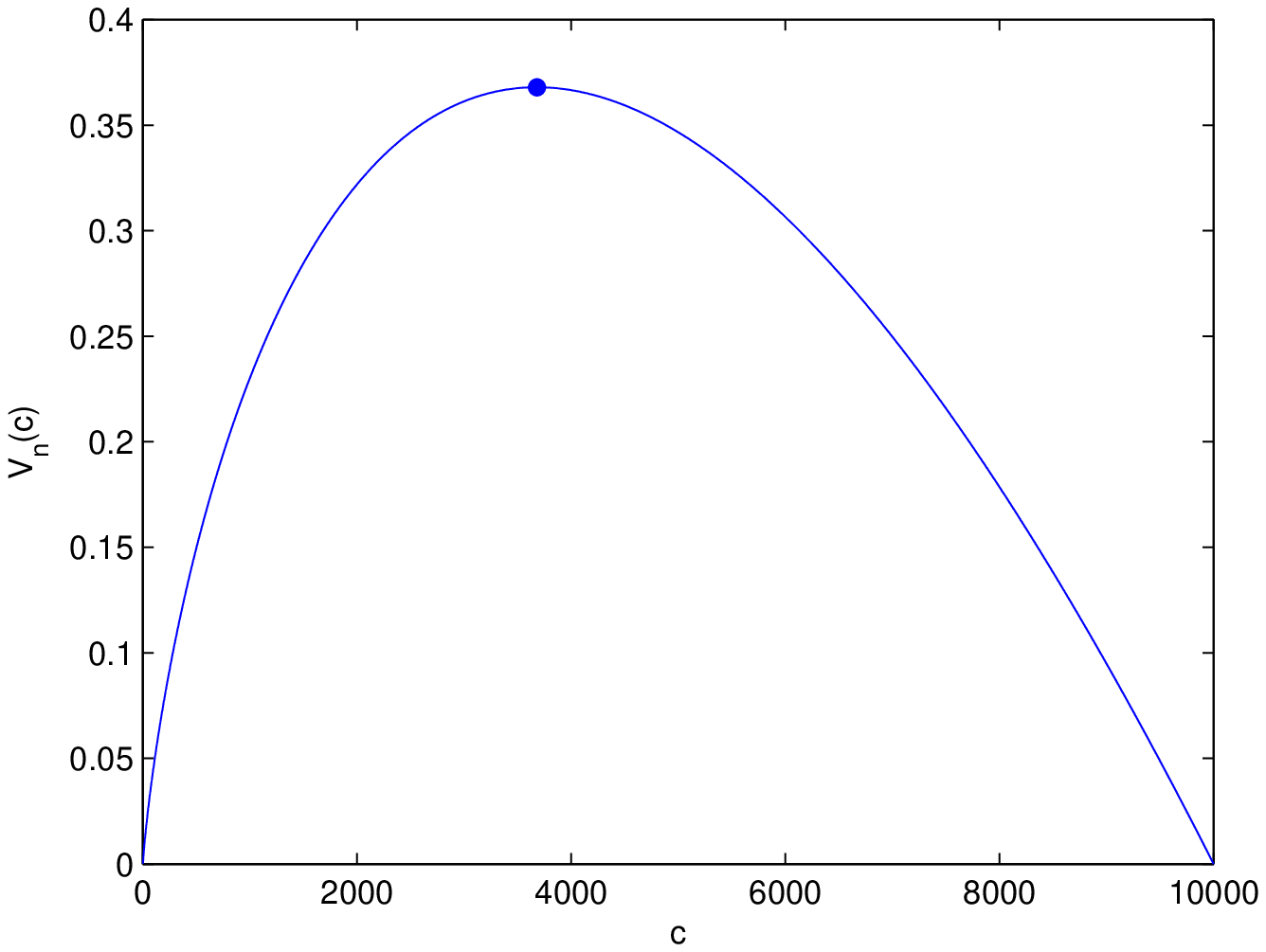} }\\
\multicolumn{2}{c}{
{\footnotesize \centering (c) Classical problem.}
} \\
  \includegraphics[width=55mm]{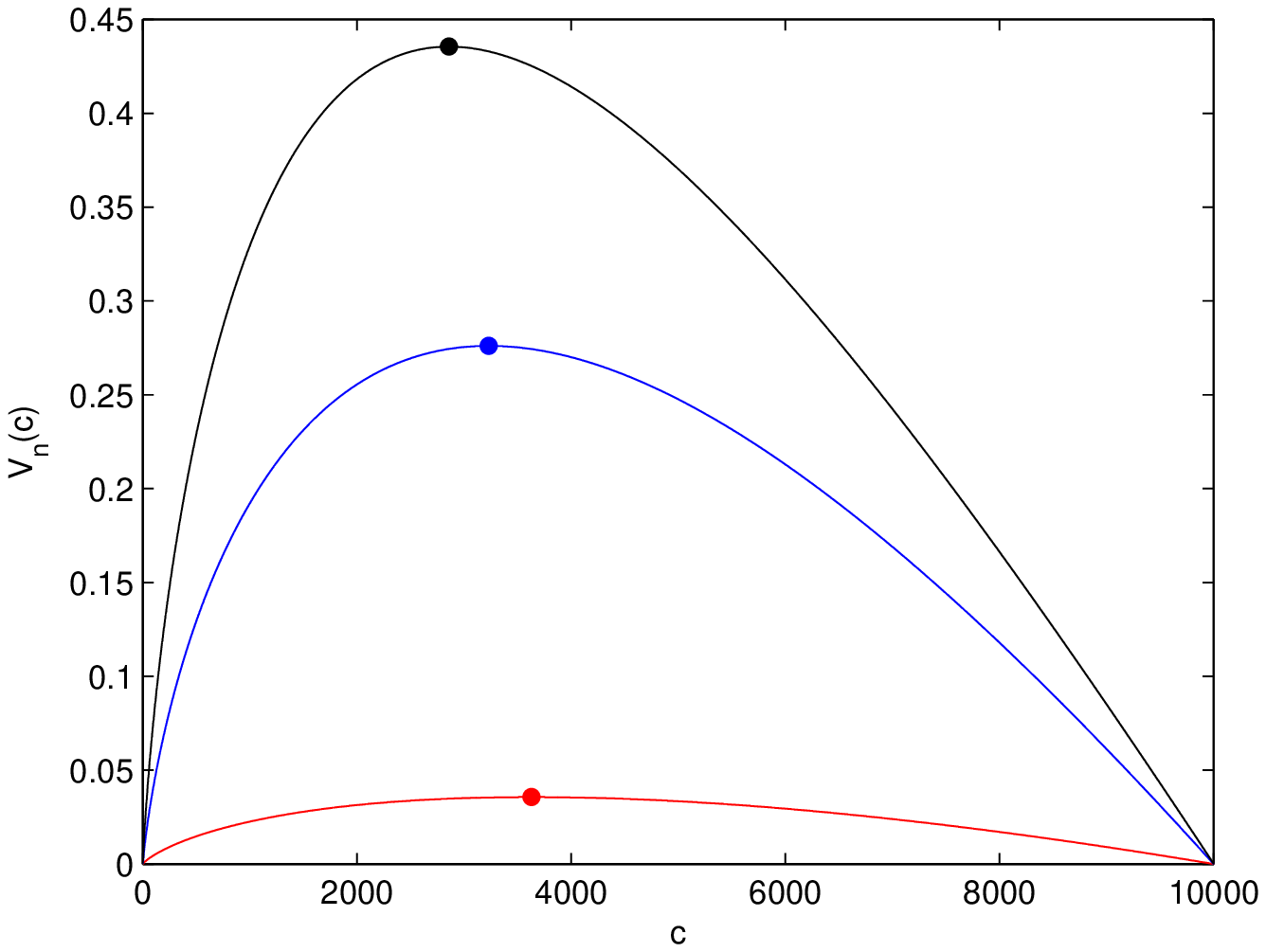} &   \includegraphics[width=55mm]{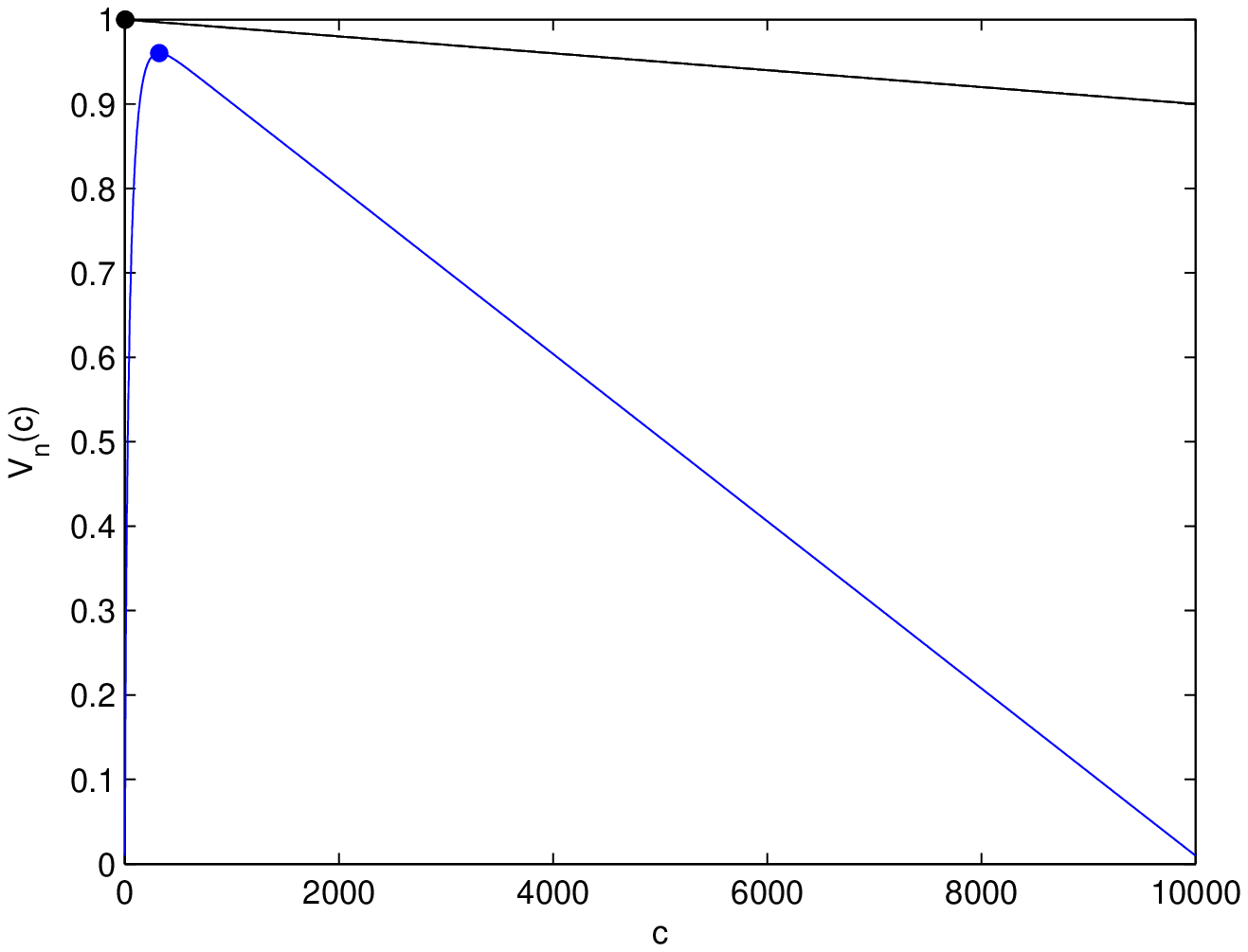} \\
{\footnotesize \centering (d) Bernoulli distribution with $(1-p)n$ close to $\alpha\ge 0$.} & 
{\footnotesize \centering (e) Bernoulli distribution with large values of $(1-p)n$.}\\
\end{tabular}
\caption{Plots of $V_n$ as a function of $c$ for different distributions with $n=10000$. The maxima are highlighted.
(a) In example \ref{ex:exponential}, the maximum is attained at $c=1022$, whereas $\frac{n}{\ln(n)+\gamma} \sim 1021.7$.
(b) Example \ref{ex:pareto}, with $\alpha=1.5$ (black) and $\alpha=1+10^{-10}$ (blue), normalized in
order to make $V_n(c_*)=1.$ The points where the maxima are attained approach to $n/e$ as $\alpha$ tends to 1.
(c) In Example \ref{ex:classic}, the optimal threshold index is $c=3680\sim n/e$.
Plots (d) and (e) are taken from Example \ref{ex:bernoulli}.
(d) The colors correspond to $p=1-2/n$ (black), $p=1-1/n$ (blue) and $p=1-0.1/n$ (red). 
The points where the maxima are attained approach to $n/e$ as $p$ tends to 1.
(e) Plots for $p=0.1$ (black) and $p=0.99$ (blue). 
The points where the maxima are attained tend to $2$ as $p \to 0.$ 
}
\label{fig:graficos}
\end{figure}

\subsection{Continuous i.i.d. random variables}\label{ss:iidcont}
Next we give three examples of i.i.d. continuous random variables that will suggest 
that `the heavier the weight of the tails, the longer the recruiter has to wait'.
Let us recall that for uniform distributions it holds that $c_*(n) \sim \sqrt{n}$ (cf.~\cite{Bea}); here we provide other examples of interest.

\begin{example}[Exponential distribution]\label{ex:exponential}
Let the candidates' values
 be given by independent exponential distributions,  $X_k\sim \exp(1)$.  It is easy to verify that
\begin{align*}
\mu_k=\sum_{j=1}^k\frac{1}{j}, 
\end{align*}
thus,
\begin{align*}
\mu_{k+1}-\mu_k=\frac{1}{k+1}.
\end{align*}
Therefore, formula~\eqref{eq:derivada2} becomes
\begin{align*}
\begin{split}
\fd{V}_n(c) 
& =\sum_{k=c}^{n-1}\frac{1}{k(k+1)}-\frac{\sum_{k=1}^n\frac{1}{k}-1}{n-1}
\\
&=
\frac{1}{c}-\frac{1}{n}+\frac{1}{n-1}-\frac{\sum_{i=1}^{n-1}\frac{1}{i}}{n-1}. \\
\end{split}
\end{align*}
From this last equation it is straightforward to check that $$ c_*(n) \sim \frac{n}{\ln(n)+\gamma} $$
where $$ \gamma = \lim_n \sum_{k=1}^{n} \frac1k - \ln n $$ is the Euler-Mascheroni constant.
Indeed, setting $\fd{V}_n(c)\le 0$, we immediately bound
\[ 
0 > \frac{1}{c} -\frac{\sum_{i=1}^{n-1}\frac{1}{i}}{n-1} ,
\]
so that $$ c > \frac{n-1}{\sum_{i=1}^{n-1}\frac{1}{i}} $$ from which 
$$ c_*(n) > \left \lfloor \frac{n-1}{\sum_{i=1}^{n-1}\frac{1}{i}} \right \rfloor. $$
The bound $$ c_*(n) < \left \lceil \frac{n}{\sum_{i=1}^{n}\frac{1}{i}} \right \rceil $$ follows 
similarly.
\end{example}

\begin{example}[Normal distribution]
Assume the aspirants' qualities 
are given by a distribution $N(0,1)$. Then, the expected value of the maximum among the first $k$ applicants 
satisfies
\[ 
\mu_k^2  = \ln \left(\frac{k^2}{2 \pi}\right) - \ln \ln \left(\frac{k^2}{2\pi} \right) + f(k),
\]
where $f$ is a function such that $\lim_{k \to \infty} f(k) = 4\gamma$ (see \cite{DN}*{Ex. 10.5.3}).
Therefore, if $k$ is large enough,
\begin{equation}
\ln \left(\frac{k^2}{2 \pi}\right) - \ln \ln \left(\frac{k^2}{2\pi} \right) \le \mu_k^2 \le
\ln \left(\frac{k^2}{2 \pi}\right).
\label{eq:max_gaussianas} 
\end{equation}
This last inequality allows us to obtain an upper bound for the optimal threshold index. Recall that, due to Proposition \ref{prop:c_diverge}, 
$c_*(n) \to \infty$ as $n\to \infty$.

Resorting to \eqref{eq:max_gaussianas} we write
\[
\mu_c \fd{V}_n(c) \le \sum_{k=c+1}^{n-1} \frac{\mu_k^2}{k(k-1)} - \frac{\mu_c^2}{c} \le 
\sum_{k=c+1}^{n-1} \frac{\ln \left(\frac{k^2}{2 \pi}\right)}{k(k-1)} - \frac{\ln \left(\frac{c^2}{2 \pi}\right)}{c} + \frac{\ln \ln \left(\frac{c^2}{2\pi} \right)}{c}.
\]
Observe that 
\begin{equation*}
\sum_{k=c+1}^{n-1} \frac{\ln \left(\frac{k^2}{2 \pi}\right)}{k(k-1)} 
\le \int_c^{n-1} \frac{\ln \left(\frac{t^2}{2 \pi}\right)}{(t-1)^2} \, dt,
\end{equation*}
whose right hand side equals to
\begin{equation*}
\frac{\ln \left(\frac{c^2}{2\pi}\right)}{c-1} + 2 \ln \left(\frac{c}{c-1} \right)
- \frac{\ln \left(\frac{(n-1)^2}{2\pi}\right)}{n-2} - 2 \ln \left(\frac{n-1}{n-2} \right).
\end{equation*}
Thus, since $$ \frac{\ln \left(\frac{c^2}{2\pi}\right)}{c}<1 \qquad \text{ and } \qquad
 \ln \left(\frac{c}{c-1} \right) < \frac{1}{c-1}, $$ 
we bound
\begin{align*}
\mu_c & \fd{V}_n(c)  \le \\
& \le 2 \ln \left(\frac{c}{c-1}\right) +\frac{\ln \left(\frac{c^2}{2\pi}\right)}{c(c-1)} + \frac{\ln\ln \left(\frac{c^2}{2\pi}\right)}{c} - 
\frac{\ln \left(\frac{(n-1)^2}{2\pi}\right)}{n-2} - 2 \ln \left(\frac{n-1}{n-2}\right) \\
& < \frac{3 + \ln\ln \left(\frac{c^2}{2\pi}\right)}{c-1} - \frac{\ln \left(\frac{(n-1)^2}{2\pi}\right)}{n-2} .
\end{align*}

Finally, if we substitute $ c $ by
$$ \tilde{c} = \frac{(n-2) \ln \ln n}{\ln\left(\frac{(n-1)^2}{2\pi}\right)}+1 $$ 
in the expression above, we obtain
\[
\fd{V}_n(\tilde c) < \frac{\ln \left(\frac{(n-1)^2}{2\pi}\right)}{\mu_{\tilde c} (n-2)} 
\left(\frac{3 + \ln\ln \left(\frac{\tilde c^2}{2\pi}\right)}{\ln \ln n} - 1 \right) \le 0,
\]
for $n$ large enough. This provides an upper bound for the optimal 
threshold index, namely
\[
c_*(n) \le \frac{(n-2) \ln \ln n}{\ln\left(\frac{(n-1)^2}{2\pi}\right)}+1.
\]
\end{example}

\begin{example}[Pareto distribution]\label{ex:pareto}
Consider $\pare{X_{k}}_{k=1,\ldots,n}$ i.i.d. Pareto distributions with CDF equal to
$F(x ;\alpha)=1-x^{-\alpha}$ for  $x \ge 1$,
where $\alpha >1$. In this case we have~\cite{DN}*{pg. $52$} that
\[ \mu_{k}= \frac{ \Gamma\pare {1-\frac{1}{\alpha}} \Gamma(k+1)}{\Gamma\pare{k+1-\frac{1}{\alpha}}}.  \]

In \cite{Gau} it is proved the estimates 
\[ (k+1)^{\frac{1}{\alpha}} \ge \frac{\Gamma(k+1)}{\Gamma\pare{k+1-\frac{1}{\alpha}}} \ge k^{\frac{1}{\alpha}}  .\]
Recall that $\Gamma(x+1)=x\Gamma(x)$, and so
\[ \mu_{k+1}-\mu_{k}= \frac{\mu_{k+1}}{\alpha (k+1)} \ge \frac{\Gamma\pare{1-\frac{1}{\alpha}}(k+1)^{\frac{1}{\alpha}-1}}{\alpha} .\]
Thus, taking into account \eqref{eq:derivada2}, we obtain
\[
\begin{split}
\fd{V}_n(c) &\ge  \Gamma\pare{1-\frac{1}{\alpha}}
\pare{-\frac{(n+1)^{\frac{1}{\alpha}}-1}{n-1} + \sum_{k=c}^{n-1}  \frac{(k+1)^{\frac{1}{\alpha}-2}}{\alpha} } 
\\ 
&\ge \Gamma\pare{1-\frac{1}{\alpha}}
\pare{-\frac{(n+1)^{\frac{1}{\alpha}}-1}{n-1} + \frac{1}{\alpha} \int_{c+1}^{n+1} x^{\frac{1}{\alpha}-2} dx }.
\end{split} \]

Note that for fixed $\alpha$, when $n$ is large enough, we have the inequality
\[ -\frac{(n+1)^{\frac{1}{\alpha}}-1}{n-1} \ge -(n+1)^{\frac{1}{\alpha}-1} , \] 
therefore we can write
\begin{align*}
\fd{V}_n(c) \ge  \Gamma\pare{1-\frac{1}{\alpha}} \pare{(c+1)^{\frac{1}{\alpha}-1} \frac{1}{\alpha-1} - (n+1)^{\frac{1}{\alpha}-1} \pare{\frac{1}{\alpha-1}+ 1}}.
\end{align*}

Thus, we obtain that if $$ \frac{n+1}{c+1} \ge  {\alpha}^{\frac{\alpha}{\alpha-1}}   $$ 
then $\fd{V}_n(c) \ge 0$. This implies that
\[\frac{n+1}{c_*(n)+1} \le {\alpha}^{\frac{\alpha}{\alpha-1}}\]
whenever $n$ large enough. 
Finally, note that this last term tends to $e$ as $\alpha$ approaches $1$, thus
the bound of Theorem \ref{teo:bound} is asymptotically sharp for the Pareto distribution when $\alpha \to 1.$
\end{example}

\subsection{Discrete random variables}\label{ss:disc}
In this paragraph we study  the behavior of several discrete random variables. 
We also recover the solution of the classical Secretary Problem and provide a family of examples
for which, as certain parameter varies, the threshold index $ c_* $ exhibits both extremal behaviors, 
the linear one limited by the upper bound from Theorem~\ref{teo:bound} and 
the constant one attained by the minimum possible of $ c_*(n) = 2 $ (see Remark~\ref{r:extbeh} below).  

\begin{example}[Permutations]
If the candidates' values $(X_k)_{1\leq k \leq n}$ are distributed uniformly over the permutations of 
the numbers from  $1$ to $n$, these random variables are exchangeable and satisfy 
\begin{align*}
\mu_k={n \choose k}^{-1} \sum_{i=k}^n i{i-1 \choose k-1} .
\end{align*}
Since
\begin{align*}
\sum_{i=k}^n i {i-1\choose k-1} = k{n+1\choose k+1},
\end{align*}
we obtain the simpler expression
\begin{align*}
\mu_k&={n \choose k}^{-1}  k{n+1\choose k+1}
=\frac{k(n+1) }{k+1},
\end{align*}
and then conclude
\begin{align*}
\mu_{k+1}-\mu_k=\frac{n+1}{(k+2)(k+1)}.
\end{align*}
This, together with~\eqref{eq:derivada2}, 
allows us to estimate the discrete derivative of $V_n$,
\begin{align*}
\fd{V}_n(c) 
=\sum_{k=c}^{n-1}\frac{n+1}{k(k+1)(k+2)}-\frac{n-\frac{n+1}{2}}{n-1}
\end{align*}
and hence
\begin{align*}
\fd{V}_n(c)=
\frac{n^2+n-c^2-c}{2 c (c+1) n}-\frac{1}{2} 
= 
\frac{n+1}{2 c (c+1)}-\frac{1}{2n}-\frac{1}{2}
\end{align*}
from which
\begin{align*}
c_*(n)=\left \lfloor\sqrt{ n+1/4}+1/2\right \rfloor
\end{align*}
is easily obtained.

It is worth noting that this kind of behavior is expected, as this situation is a discrete analogue of the uniform distribution.
Even more, it can be deduced from the considerations in \cite{Bea} that $\left \lfloor\sqrt{ n+1/4}+1/2\right \rfloor$ is also the exact threshold index in the uniform distribution case.
\end{example}

\begin{example}[Recovering the classical Secretary Problem] \label{ex:classic}
Consider random variables  $\pare{X_{k}}_{ 1 \le k \le n }$ in such a way that for each $k$, with probability $1/n$ it holds that $X_{k}$ is equal to $1$ and the rest of them are equal to $0$.
This is equivalent to the classical Secretary Problem since, in this case, maximizing the expected value corresponds to maximizing the probability of selecting the best applicant. Since $ \mu_k= \frac{k}{n} ,$
applying equation \eqref{eq:derivada2} we obtain
\[ \fd{V}_n(c)=\frac{1}{n} \pare{\sum_{k=c}^{n-1} {\frac{1}{k}} -1 } ,\]
Therefore, $c_*(n)$ is the least integer $c$ that satisfies $\sum_{k=c}^{n-1} {\frac{1}{k}} \le 1$, as in the classical Secretary Problem.
\end{example}

\begin{example}[Bernoulli variables]\label{ex:bernoulli}
Assume the applicants' qualities
 are given by i.i.d. Bernoulli random variables $X_k \sim B(1,1-p)$, so that $\mu_k=1-p^k.$
Recalling 
formula~\eqref{eq:derivada2}, it follows that
\begin{align*}
\fd{V}_n(c)
=\sum_{k=c}^{n-1}\frac{p^k-p^{k+1}}{k}-\frac{p-p^n}{n-1} 
=(1-p)\sum_{k=c}^{n-1}\frac{p^k}{k}-\frac{p-p^n}{n-1}.
\end{align*}
Since the function $k \mapsto \frac{p^k}{k}$ is decreasing, after a change of variables we obtain
\[
\int_{c/n}^1 \frac{p^{nz}}{z}  dz  \le
\sum_{k=c}^{n-1}\frac{p^k}{k} \le 
\int_{(c-1)/(n-1)}^1 \frac{p^{(n-1)z}}{z}  dz.
\] 

We study the behavior of the optimal threshold index in two different scenarios.
For this purpose, let us consider $p=p(n)$ and define
$$ f(n):=(1-p)n, $$ 
the expected number of candidates with quality $X_k=1$.
The two aforementioned situations are distinguished by the asymptotic behavior of $f(n)$ as $n\to \infty$.

In first place, assume that $\lim_n f(n) = \alpha \ge 0$ and perform calculations for $c=c(n)$. In such case, we have that $\lim_n p^n = e^{-\alpha}$ and
\[
\lim_n \frac{p-p^n}{(n-1)(1-p)} = \begin{cases}
 \frac{1-e^{-\alpha}}{\alpha} &  \text{ if } \alpha > 0, \\
 \quad 1 & \text{ if } \alpha =0.
\end{cases}
\]
Let us call $g(\alpha)$ the function defined by the right hand side above. Then, given $\epsilon>0$, if $n$ is large enough we obtain
\[
\frac{\fd{V}_n(c)}{1-p} \ge  \int_{c/n}^1 \frac{p^{nz}}{z}  dz - g(\alpha) - \epsilon.
\]
As the integral above is non convergent if the lower limit $c/n$ is substituted by $0$, if $c/n \to 0$ we have $\fd{V}_n(c) \ge 0$. This shows that the optimal threshold index is linear in $n$. A lower bound may be obtained if $c/n = \beta > 0$, since due to the Dominated Convergence Theorem,
\[
\lim_n \int_{c/n}^1 \frac{p^{nz}}{z}  dz = \int_\beta ^1 \frac{e^{-\alpha z}}{z} dz .
\]
Therefore,
\[
\fd{V}_n(n\beta) \ge \int_\beta ^1 \frac{e^{-\alpha z}}{z} dz - g(\alpha) - 2\epsilon 
\]
if $n$ is large enough. Taking $\beta$ such that
\[
\int_\beta ^1 \frac{e^{-\alpha z}}{z} dz = g(\alpha),
\]
we obtain a lower bound for $c_*$. For example, if $\alpha = 0$ we recover the sharp bound $\beta = 1/e$, and for $\alpha=1$ we obtain $\beta \simeq 0.323.$

On the other hand, assuming that $\lim_n f(n) = \infty$ it is possible to show that the optimal threshold index is not linear in $n$. Indeed, simple calculations give
\begin{equation} \label{eq:bernoulli}
\fd{V}_n(c) \le \frac{1}{n-c} \int_{(c-1)/(n-1)}^1 \left[ (1-p)(n-c) \frac{p^{(n-1)z}}{z} - p + p^n \right] dz.
\end{equation}
Assume that there exists an $\epsilon > 0$ such that $(c-1)/(n-1) > \epsilon$. Then, the integrand is easily shown to be negative for large $n$, because
\[
(1-p)(n-c) \frac{p^{(n-1)z}}{z} - p + p^n \le p \left[ \frac{(1-p)(n-c)(n-1) p^{\epsilon (n-1)-1}}{c-1} - 1 + p^{n-1} \right] \le 0
\] 
for large $n$, because the term in brackets tends to $-1$ as $n \to \infty$.
This shows that if $c$ is linear in $n$, the derivative at $c$ is negative. Thus, the optimal threshold index is not linear in $n$.
\end{example}

\begin{remark}\label{r:extbeh}
This last example provides a family of situations for which there is no lower bound for the optimal threshold index. Indeed, let us fix $n$ and $c > 2$ and consider the limit $p\to 0$. The integrand in the right hand side of\eqref{eq:bernoulli} is pointwise bounded above by
\[
p \left( \frac{(1-p)(n-c)(n-1) p^{c-2}}{c-1} - 1 + p^{n-1}  \right),
\]
and so it is negative for every $z \in [(c-1)/(n-1), 1]$ if $p$ is small enough. This proves that
$\fd{V}_n(c) \le 0$ for every $n,c > 2$ if $p$ is small enough, and thus $c_* =2$.
\end{remark}

\bibliography{bib_secretary}

\begin{bibdiv}
\begin{biblist}

\bib{Bea}{article}{
      author={Bearden, J.~Neil},
       title={A new secretary problem with rank-based selection and cardinal
  payoffs},
        date={2006},
        ISSN={0022-2496},
     journal={J. Math. Psych.},
      volume={50},
      number={1},
       pages={58\ndash 59},
         url={http://dx.doi.org/10.1016/j.jmp.2005.11.003},
      review={\MR{2208064}},
}

\bib{Bru}{article}{
      author={Bruss, F.~Thomas},
       title={Sum the odds to one and stop},
        date={2000},
        ISSN={0091-1798},
     journal={Ann. Probab.},
      volume={28},
      number={3},
       pages={1384\ndash 1391},
         url={http://dx.doi.org/10.1214/aop/1019160340},
      review={\MR{1797879}},
}

\bib{Bruss}{article}{
      author={Bruss, F.~Thomas},
       title={What is known about robbins' problem?},
        date={2005},
        ISSN={00219002},
     journal={Journal of Applied Probability},
      volume={42},
      number={1},
       pages={108\ndash 120},
         url={http://www.jstor.org/stable/30040773},
}

\bib{Chow}{article}{
      author={Chow, Y.S.},
      author={Moriguti, S.},
      author={Robbins, H.},
      author={Samuels, S.M.},
       title={Optimal selection based on relative rank (the ``secretary
  problem'')},
        date={1964},
     journal={Israel Journal of mathematics},
      volume={2},
      number={2},
       pages={81\ndash 90},
}

\bib{DN}{book}{
      author={David, H.~A.},
      author={Nagaraja, H.~N.},
       title={Order statistics},
     edition={Third},
      series={Wiley Series in Probability and Statistics},
   publisher={Wiley-Interscience [John Wiley \& Sons], Hoboken, NJ},
        date={2003},
        ISBN={0-471-38926-9},
         url={http://dx.doi.org/10.1002/0471722162},
      review={\MR{1994955}},
}

\bib{Fer}{article}{
      author={Ferguson, Thomas~S.},
       title={Who solved the secretary problem?},
        date={1989},
        ISSN={0883-4237},
     journal={Statist. Sci.},
      volume={4},
      number={3},
       pages={282\ndash 296},
  url={http://links.jstor.org/sici?sici=0883-4237(198908)4:3<282:WSTSP>2.0.CO;2-S&origin=MSN},
        note={With comments and a rejoinder by the author},
      review={\MR{1015277}},
}

\bib{Gau}{article}{
      author={Gautschi, Walter},
       title={Some elementary inequalities relating to the gamma and incomplete
  gamma function},
        date={1959/60},
        ISSN={0097-1421},
     journal={J. Math. and Phys.},
      volume={38},
       pages={77\ndash 81},
      review={\MR{0103289}},
}

\bib{GilMos}{article}{
      author={Gilbert, John~P.},
      author={Mosteller, Frederick},
       title={Recognizing the maximum of a sequence},
        date={1966},
        ISSN={01621459},
     journal={Journal of the American Statistical Association},
      volume={61},
      number={313},
       pages={35\ndash 73},
         url={http://www.jstor.org/stable/2283044},
}

\bib{Gird}{inproceedings}{
      author={Girdhar, Yogesh},
      author={Dudek, Gregory},
       title={Optimal online data sampling or how to hire the best
  secretaries},
        date={2009},
   booktitle={Proceedings of the 2009 canadian conference on computer and robot
  vision},
      series={CRV '09},
   publisher={IEEE Computer Society},
     address={Washington, DC, USA},
       pages={292\ndash 298},
         url={http://dx.doi.org/10.1109/CRV.2009.30},
}

\bib{Lin}{article}{
      author={Lindley, D.~V.},
       title={Dynamic programming and decision theory},
        date={1961},
        ISSN={00359254, 14679876},
     journal={Journal of the Royal Statistical Society. Series C (Applied
  Statistics)},
      volume={10},
      number={1},
       pages={39\ndash 51},
         url={http://www.jstor.org/stable/2985407},
}

\bib{matsui2016lower}{article}{
      author={Matsui, Tomomi},
      author={Ano, Katsunori},
       title={Lower bounds for {B}russ' odds problem with multiple stoppings},
        date={2016},
     journal={Mathematics of Operations Research},
      volume={41},
      number={2},
       pages={700\ndash 714},
}

\bib{tamaki2010sum}{article}{
      author={Tamaki, Mitsushi},
       title={Sum the multiplicative odds to one and stop},
        date={2010},
     journal={Journal of Applied Probability},
      volume={47},
      number={3},
       pages={761\ndash 777},
}

\bib{Van}{article}{
      author={Vanderbei, R.~J.},
       title={The optimal choice of a subset of a population},
        date={1980},
        ISSN={0364-765X},
     journal={Math. Oper. Res.},
      volume={5},
      number={4},
       pages={481\ndash 486},
         url={http://dx.doi.org/10.1287/moor.5.4.481},
      review={\MR{593639}},
}

\bib{vanderbei1995postdoc}{article}{
      author={Vanderbei, Robert~J},
       title={The postdoc variant of the secretary problem},
        date={1995},
     journal={Technical Reviews, Princeton},
}

\end{biblist}
\end{bibdiv}
\bibliographystyle{plain}
\end{document}